%% file: manuscript.tex
\documentclass[
  a4paper,oneside,DIV=12,
  12pt,
  headsepline,
  ]{scrartcl}

\input{packages}
\input{layout}

\input{shortcuts}
\input{shortcuts_this}

\title{%
  Weighted Recursions for the \\ Smallest Parts Function
}
\author{%
  Matthew Ortiz%
}
\newcommand{\headertitle}{%
  Weighted Recursions for the Smallest Parts Function
}
\newcommand{\headerauthors}{%
  Matthew Ortiz
}
\ifdraft{\date{\today\ at\ \thistime}}{\date{}}

\begin{document}

\thispagestyle{scrplain}
\begingroup
\deffootnote[1em]{1.5em}{1em}{\thefootnotemark}
\maketitle
\endgroup

\begin{abstract}
\small
\noindent
{\tbf Abstract:}
We establish new polynomial-weighted recursions for Andrews' smallest parts function.
Our results use the generating series for the spt function, a harmonic Maass form of weight 3/2, paired with the Dedekind eta function via the Rankin-Cohen bracket.
In weights larger than 2, we find a nontrivial quasimodular component, which we determine for the relevant weights.
We apply the holomorphic projection operator and the vanishing of cusp form spaces of small enough weight to obtain our results.
  \\[.3\baselineskip]
  \noindent
  \textsf{\textbf{%
      Rankin--Cohen brackets%
    }}%
  \noindent
  \ {\tiny$\blacksquare$}\ %
  \textsf{\textbf{%
      Holomorphic projection%
    }}%
  \noindent
  \ {\tiny$\blacksquare$}\ %
  \textsf{\textbf{%
      Quasimodular forms%
    }}
  \\[.2\baselineskip]
  \noindent
  \textsf{\textbf{%
      MSC Primary: 11F37, 11F30
    }}%
\end{abstract}

\Needspace*{4em}
\phantomsection
\label{sec:introduction}
\addcontentsline{toc}{section}{Introduction}
\markright{Introduction}
First introduced by Andrews in~\cite{andrews-2008}, the smallest parts function, denoted~$\spt(n)$, counts the occurrence of the smallest part in every partition of~$n$.
For example, the partitions of 4 are
\begin{gather*}
  4=\ul{4}
  =3+\ul{1}
  =\ul{2}+\ul{2}
  =2+\ul{1}+\ul{1}
  =\ul{1}+\ul{1}+\ul{1}+\ul{1}
  \tx{,}
\end{gather*}
with the smallest parts underlined.
Hence~$\spt(4)=10$.

As with the classical partition function, number-theoretic properties of~$\spt$ have been extensively studied; see for example~\cite{garvan-2012}.
A notable result, recorded by Zagier~\cite{zagier-2009}, is the recursion
\begin{gather*}
  \sum_{m^2+3m\le 2N}
  \hspace{-12pt}
  (-1)^m
  \,
  \spt\big(
  N-
  \mfrac{m^2+3m}{2}
  \big)
  =
  a(N)
  \tx{,}
  \quad
  \tx{where}
  \quad
  a(N)
  :=
  -
  \sum_{\substack{ab=6N \\ a<b}}
  \big(
  \mfrac{12}{b^2-a^2}
  \big)
  \cdot a
  \tx{.}
\end{gather*}
This identity follows from the appearance of the functions in coefficients of a harmonic Maass form, combined with holomorphic projection.
A vector-valued analog of this idea is used to derive recursions for some of Ramanujan's mock theta functions in~\cite{imamoglu-raum-richter-2014}.

The aforementioned recursions are unweighted (or linearly weighted in the case of~\cite{imamoglu-raum-richter-2014}), and come from considering weight 2 modular forms.
By contrast, the present work examines modular forms of higher weight, yielding the following recursions with higher-degree polynomial weights.
To state these, define the divisor sum
\begin{gather}
  \label{eq:lambda_star}
  \tag{DS}
  \lambda_{k}^*
  (N)
  =
  \sum_{\substack{ab=6N \\ a\equiv 0\bmod{3} \\ b\not\equiv 0\bmod{3} \\ a<b}}
  (-1)^{a/3}
  a^{k}
  +
  \sum_{\substack{ab=6N \\ a\not\equiv 0\bmod{3} \\ b\equiv 0\bmod{3} \\ a<b}}
  (-1)^{b/3}
  a^{k}
  \tx{.}
\end{gather}

\begin{maintheorem}
  \label{thm:main_theorem}
  We have the recursions
  \begin{align*}
    \sum_{3m^2+m\le 2N}
    \hspace{-12pt}
    (-1)^m
    \;
    F_1(N,m)
    \;
    \spt\big(
    N-\mfrac{3m^2+m}{2}
    \big)
    &=
    -\lambda_3^*(N)
    -12N\sigma_1(N)
    +20\sigma_3(N)
    \tx{,}
    \\
    \sum_{3m^2+m\le 2N}
    \hspace{-12pt}
    (-1)^m
    \;
    F_2(N,m)
    \;
    \spt\big(
    N-\mfrac{3m^2+m}{2}
    \big)
    &=
    -\lambda_5^*(N)
    -72N^2\sigma_1(N)
    +60N\sigma_3(N)
    -42\sigma_5(N)
    \tx{,}
    \\
    \sum_{3m^2+m\le 2N}
    \hspace{-12pt}
    (-1)^m
    \;
    F_3(N,m)
    \;
    \spt\big(
    N-\mfrac{3m^2+m}{2}
    \big)
    &=
    -\lambda_7^*(N)
    -432N^3\sigma_1(N)
    +360N^2\sigma_3(N)
    \\
    &\quad
    -84N\sigma_5(N)
    +40\sigma_7(N)
    \tx{,}
    \\
    \sum_{3m^2+m\le 2N}
    \hspace{-12pt}
    (-1)^m
    \;
    F_4(N,m)
    \;
    \spt\big(
    N-\mfrac{3m^2+m}{2}
    \big)
    &=
    -\lambda_9^*(N)
    -2592N^4\sigma_1(N)
    +2160N^3\sigma_3(N)
    \\
    &\quad
    -504N^2\sigma_5(N)
    +60N\sigma_7(N)
    -22\sigma_9(N)
    \tx{,}
  \end{align*}
  where
  \begin{align*}
    F_1(N,m)
    &=
    36m^2+12m-6N+1
    \\
    F_2(N,m)
    &=
    1296m^4
    +
    864m^3
    -
    648Nm^2
    +
    216m^2
    -
    216Nm
    +
    24m
    +
    36N^2
    -
    18N
    +1
    \\
    F_3(N,m)
    &=
    46656m^6
    +
    46656m^5
    -
    38880Nm^4
    +
    19440m^4
    -
    25920Nm^3
    +
    4320m^3
    +
    7776N^2m^2
    \\
    &\quad
    -
    6480Nm^2
    +
    540m^2
    +
    2592N^2m
    -
    720Nm
    +
    36m
    -
    216N^3
    +
    216N^2
    -
    30N
    +1
    \\
    F_4(N,m)
    &=
    1679616m^8
    +
    2239488m^7
    -
    1959552Nm^6
    +
    1306368m^6
    -
    1959552Nm^5
    +
    435456m^5
    \\
    &\quad
    +
    699840N^2m^4
    -
    816480Nm^4
    +
    90720m^4
    +
    466560N^2m^3
    -
    181440Nm^3
    +
    12096m^3
    \\
    &\quad
    -
    77760N^3m^2
    +
    116640N^2m^2
    -
    22680Nm^2
    +
    1008m^2
    -
    25920N^3m
    +
    12960N^2m
    \\
    &\quad
    -
    1512Nm
    +
    48m
    +
    1296N^4
    -
    2160N^3
    +
    540N^2
    -
    42N
    +
    1
  \end{align*}
\end{maintheorem}

The proof of Theorem~\ref{thm:main_theorem} uses the framework of~\cite{mertens-2016,ortiz-raum-richter-2026}, which yields formulas for the coefficients of the Rankin--Cohen bracket of a mock modular form paired with a modular form.
This lets us leverage the vanishing of cusp forms for the full modular group~$\SL{2}(\ZZ)$ of small enough weight.
See~\eqref{eq:preliminaries:identity} for an effective summary of this idea.
The main issue to account for is the completion to a harmonic Maass form containing an extra quasimodular component, which we work out explicitly for small weights.
This is the second term on the left-hand side of~\eqref{eq:preliminaries:identity}.

This paper is structured as follows. In Section~\ref{sec:preliminaries} we recall the relevant ideas, including spaces of quasimodular forms.
We also state the completion of~$\spt$ to a harmonic Maass form.
We then apply the Rankin--Cohen bracket and holomorphic projection, computing the Fourier coefficients of the result in Section~\ref{sec:fourier}.
In Section~\ref{sec:proof_main}, we explain how these computations yield the formulas in Theorem~\ref{thm:main_theorem}.

\section{Preliminaries}
\label{sec:preliminaries}
\subsection{Modular forms and Harmonic Maass forms}
\label{ssec:preliminaries:modular_forms}
For background on modular and Maass forms, see~\cite{bringmann-folsom-ono-rolen-2018}.

We first set some notation.
Let~$\mathbb{H}$ denote the upper half-plane.
For functions periodic on~$\mathbb{H}$, we write~$c(f,n;y)$ for its~$n$th Fourier coefficient and~$e(n\tau):=e^{2\pi in\tau}$ the associated power, where~$n\in\QQ$.
When the coefficient is independent of~$y$, we simply write~$c(f,n)$.
Finally,~$\Ga$ denotes a congruence subgroup of the modular group~$\SL{2}(\ZZ)$ unless noted otherwise.

We denote the spaces of modular resp. Maass forms of weight~$k$, level~$\Ga$ and multiplier system~$v$ by~$\rmM_k(\Ga,v)$ resp.~$\bbM_k(\Ga,v)$.
When~$v$ is trivial, we write~$\rmM_k(\Ga)$ resp.~$\bbM_k(\Ga)$, and when in addition~$\Ga=\SL{2}(\ZZ)$, we write~$\rmM_k$ resp.~$\bbM_k$.
Spaces of cusp forms are likewise written~$\rmS_k(\Ga,v)$,~$\rmS_k(\Ga)$, or~$\rmS_k$.

\begin{remark}
  \label{rem:metaplectic}
Note that in~\cite{ortiz-raum-richter-2026}, the metaplectic cover is used to define modular and Maass forms of half-integral weight.
  In order to use results from that work, we identify a multiplier system~$v$ for integral weight modular forms over~$\Ga\le\SL{2}(\ZZ)$ with characters~$\chi$ of~$\SL{2}(\ZZ)$, which in turn are in bijection with vector-valued modular forms of type~$\Ind_\Ga^{\SL{2}(\ZZ)}\chi$.
\end{remark}

A well-known example is the Dedekind eta function
\begin{gather}
  \label{eq:dedekind_eta}
  \eta(\tau)
  =
  e\big(\mfrac{1}{24}\tau\big)
  \prod_{n=1}^\infty
  (1-e(n\tau))
  =
  \sum_{m\in\ZZ}
  (-1)^m
  e\big(
  \big(\mfrac{3m^2+m}{2}+\mfrac{1}{24}\big)
  \tau
  \big)
  \tx{,}
\end{gather}
which is modular of weight~$1/2$ on~$\SL{2}(\ZZ)$ and a multiplier system~$v_\eta$ that can be written down explicitly, cf.~\cite{koehler-2011}, Section 1.3.

\paragraph{The Eichler integral.}
For~$f\in\rmM_{2-k}(\Ga,v)$, the Eichler integral of~$f$, denoted~$f^*$, is
\begin{gather}
  \label{eq:def:non-holomorphic-eichler-integral}
  \begin{aligned}
  f^\ast(\tau)
   & :=
  -(2 i)^{k-1}
  \int_{-\ov{\tau}}^{i\infty}
  \frac{\ov{f(-\ov{w})}}{(w + \tau)^k}
  \,\rmd\mspace{-2mu}w
  \\
   & \hphantom{:}=
  \frac{\ov{c(f,0)}}{1-k}\, y^{1-k}
  \,-\,
  (4 \pi)^{k-1}
  \sum_{n < 0}
  \ov{c(f, |n|)\, }|n|^{k-1}\,
  \Gamma(1-k,4 \pi |n| y) e(n\tau)
  \tx{,}
  \end{aligned}
\end{gather}
where~$\Ga$ in this case is the incomplete Gamma function.
\subsection{Quasimodular forms}
\label{ssec:preliminaries:quasimodular_forms}
Quasimodular forms are a generalization of modular forms that weakens modularity while maintaining holomorphicity (\cite{zagier-1994, kaneko-zagier-1995}).
\begin{definition}
  Let~$k$ and~$s$ be nonnegative integers.
  A holomorphic function~$f \defcol \HS \ra \CC$ is a~\emph{quasimodular form} of weight~$k$ and depth~$s$ (over~$\SL{2}(\ZZ)$) if there are holomorphic functions~$f_0,\ldots,f_s$ on~$\HS$ such that
  \begin{gather*}
    f(\ga\tau)
    =
    (c\tau+d)^k
    \sum_{i=0}^s
    \big(
    \mfrac{c}{c\tau+d}
    \big)^i
    f_i(\tau)
    \tx{.}
  \end{gather*}
\end{definition}

  The set of quasimodular forms of weight~$k$ and depth at most~$s$ is a vector space, denoted~$\wtd{\rmM}_k^{\le s}$.
  A notable example is the holomorphic Eisenstein series~$E_2\in\wtd{\rmM}_2^1$.

  Multiplication and the~$\mu$th normalized derivative~$f^{(\mu)}=(2\pi i)^{-\mu}\,\partial_\tau^\mu f$ yield new quasimodular forms:
  \begin{proposition}
    \label{prop:quasimodular_multiplication}
      If~$f\in\wtd{\rmM}_k^{\le s}$ and~$g\in\wtd{\rmM}_l^{\le t}$, then~$fg\in\wtd{\rmM}_{k+l}^{\le s+t}$.
  \end{proposition}
  \begin{proposition}
    \label{prop:quasimodular_derivative}
    If~$f\in\wtd{\rmM}_k^{\le s}$, then~$f^{(\mu)}\in\wtd{\rmM}_{k+2\mu}^{\le s+\mu}$.
  \end{proposition}
  \begin{proof}
    This is Proposition 6 of~\cite{martin-royer-2009} applied~$\mu$ times.
  \end{proof}
  The next results show that quasimodular forms can be expressed in terms of derivatives of modular forms.
  \begin{proposition}[\cite{martin-royer-2009}, Proposition 8]
    \label{prop:quasimodular_structure}
    We have the equality
    \begin{gather*}
      \wtd{\rmM}_k^{\le k/2}
      =
      \bigoplus_{\mu=0}^{k/2}
      \rmM_{k-2\mu}^{(\mu)}
      \oplus
      \CC
      E_2^{(k/2-1)}
      \tx{,}
    \end{gather*}
    where~${\rmM}_{k-2\mu}^{(\mu)}=\{f^{(i)}:f\in \rmM_{k-2\mu}\}$. In particular,
  \begin{align*}
    \wtd{\rmM}_4^{\le 2}
    &=
    \CC E_2^{(1)}
    \oplus
    \CC E_4
    \tx{,}
    \\
    \wtd{\rmM}_6^{\le 3}
    &=
    \CC E_2^{(2)}
    \oplus
    \CC E_4^{(1)}
    \oplus
    \CC E_6
    \tx{,}
    \\
    \wtd{\rmM}_8^{\le 4}
    &=
    \CC E_2^{(3)}
    \oplus
    \CC E_4^{(2)}
    \oplus
    \CC E_6^{(1)}
    \oplus
    \CC E_8
    \tx{,}
    \\
    \wtd{\rmM}_{10}^{\le 5}
    &=
    \CC E_2^{(4)}
    \oplus
    \CC E_4^{(3)}
    \oplus
    \CC E_6^{(2)}
    \oplus
    \CC E_8^{(1)}
    \oplus
    \CC E_{10}
    \tx{.}
  \end{align*}
  \end{proposition}

\begin{remark}
  Recall that~$\Delta\in \rmM_{12}$, where~$\Delta$ is the modular discriminant, and hence~$\Delta^{(1)}\in \wtd{\rmM}_{14}^{\le 7}$.
  This is the reason Theorem~\ref{thm:main_theorem} has no formula corresponding to~$\nu=6$, despite~$S_{14}=0$.
\end{remark}
\subsection{Rankin--Cohen brackets}
\label{ssec:preliminaries:rankin-cohen}
From two real-analytic modular-invariant functions, one can create a new modular-invariant function using the Rankin--Cohen bracket (\cite{rankin-1956, cohen-1975}).
\begin{definition}
\label{def:rankin-cohen}
Let~$f,g \defcol \HS \rightarrow \CC$ be real-analytic and~$k,l\in\QQ$. The \emph{$\nu$\thdash{} Rankin--Cohen bracket} of~$f$ and~$g$ is defined by
\begin{gather}
  \label{eq:def:rankin-cohen}
  \big[ f,g \big]_\nu
  \defeq
  \sum_{\mu=0}^\nu(-1)^\mu
  \mbinom{k+\nu-1}{\nu-\mu}
  \mbinom{l+\nu-1}{\mu}\;\;
  f^{(\mu)}\,
  g^{(\nu-\mu)}
  \tx{.}
\end{gather}
  This generalizes multiplication of modular forms, in the sense that~$[f,g]_0=fg$.

  Cohen (\cite{cohen-1975}, Theorem 7.1) showed that if~$f$ and~$g$ are modular of weights~$k$ and~$l$, then~$[f,g]_\nu$ is modular of weight~$k+l+2\nu$.
  In particular, if~$f\in\rmM_k(\Ga,v_1)$ and~$g\in\rmM_l(\Ga,v_2)$, then
  \begin{gather}
    \label{eq:rankin-cohen_modular_form}
    \big[f,g\big]_\nu
    \in
    \rmM_{k+l+2\nu}
    (\Ga,v_1v_2)
    \tx{.}
  \end{gather}
\end{definition}

\subsection{Holomorphic projection}
\label{ssec:preliminaries:projection}
We recall the holomorphic projection operator; see also~\cite{gross-zagier-1986, sturm-1980}.
\begin{definition}
  \label{def:holomorphic_projection}
  Let~$f\defcol \HS \ra \CC$ be continuous with Fourier expansion
  \begin{gather*}
    f(\tau)
    =
    \sum_{m\in\QQ}
    c(f,m;y)
    e(m\tau)
    \tx{,}
  \end{gather*}
  which also satisfies the following two conditions:
  \begin{enumerateroman}
    \item
      \label{it:def:holomorphic_projection:constant}
      $f(\tau) = c_0 + \cO(y^{-\epsilon})$ as~$y \rightarrow \infty$ for some~$\epsilon > 0$ and~$c_0 \in \CC$, and
    \item
      \label{it:def:holomorphic_projection:growth}
      for all~$m > 0$, $c(f, m; y) = \cO(y^{1-k+\epsilon})$ as~$y \rightarrow 0$ for some~$\epsilon > 0$.
  \end{enumerateroman}
  The holomorphic projection of~$f$ is
  \begin{gather}
    \label{eq:def:holomorphic_projection}
    \pi^\hol_{k}(f)
    \defeq
    c_0 + \sum_{0 < m \in \frac{1}{N}\ZZ}
    c(m)\, e(m \tau)
    \quad\tx{with}\quad
    c(m)
    \defeq
    \mfrac{(4\pi m)^{k-1}}{\Gamma(k-1)}\,
    \lim_{s \ra 0}\,
    \int\limits_0^\infty c(f; m; y)\, e^{-4\pi m y} y^{s+k-2} \,\intrmd y
    \tx{,}
  \end{gather}
  where~$\Ga$ denotes the Gamma function.
\end{definition}
The assumptions~\ref{it:def:holomorphic_projection:constant} and~\ref{it:def:holomorphic_projection:growth} yield convergence of the integral in~\eqref{eq:def:holomorphic_projection} at~$\infty$ and 0, respectively.
One may also show from~\eqref{eq:def:holomorphic_projection} that if~$f$ is already holomorphic, so that~$c(f,m;y)=c(f,m)$, then~$\pi^\hol_k(f)=f$; see Proposition 4 of~\cite{imamoglu-raum-richter-2014}.

Definition~\ref{def:holomorphic_projection} allows for a generalization of~\eqref{eq:rankin-cohen_modular_form}: let~$f$ and~$g$ be modular or Maass forms, with~$k$ and~$l$ now strict half-integers so that~$k+l\in\ZZ$.
From the discussion of Remark~\ref{rem:metaplectic}, Lemma 2.14 of~\cite{ortiz-raum-richter-2026} shows that for~$f$ and~$g$ as in Section~\ref{ssec:preliminaries:rankin-cohen},
\begin{gather}
  \label{eq:rankin-cohen_cusp}
  \pi^\hol_{k+l+2\nu}
  \big(
  \big[
    f,g
    \big]_\nu
  \big)
  \in
  \rmM_{k+l+2\nu}
  \big(
  \Ga,
  v_1v_2
  \big)
  \tx{;}
  \quad
  \pi^\hol_{k+l+2\nu}
  \big(
  \big[
    f,g
    \big]_\nu
  \big)
  \in
  \rmS_{k+l+2\nu}
  \big(
  \Ga,
  v_1v_2
  \big)
  \quad
  \tx{for~$\nu\ge 1$.}
\end{gather}
\subsection{spt generating series}
\label{ssec:preliminaries:spt}
It is known (\cite{ahlgren-andersen-2015}) that
\begin{gather}
  \label{eq:preliminaries:spt_maass}
  F(\tau)
  :=
  F^+(\tau)
  -
  \mfrac{1}{12}
  \mfrac{E_2(\tau)}{\eta(\tau)}
  -
  \mfrac{1}{\pi}
  \sqrt{\mfrac{3}{8}}
  \eta^*(\tau)
  \in
  \bbM_{3/2}(\SL{2}(\ZZ),\ov{v_\eta})
  \tx{,}
\end{gather}
where
\begin{gather*}
  F^+(\tau)
  =
  \sum_{n=0}^\infty
  \spt(n)
  e\big(
  \big(
  n-\mfrac{1}{24}
  \big)
  \tau
  \big)
  \tx{.}
\end{gather*}
Thus, from~\eqref{eq:rankin-cohen_cusp} and~\eqref{eq:preliminaries:spt_maass},
\begin{gather}
  \label{eq:preliminaries:identity}
  \pi_{2+2\nu}^\hol
  \big(
  [F,\eta]_\nu)
  =
  [F^+,\eta]_\nu
  -
  \mfrac{1}{12}
  [E_2/\eta,\eta]_\nu
  -
  \mfrac{1}{\pi}
  \sqrt{\mfrac{3}{8}}
  \pi_{2+2\nu}^\hol
  \big(
  [\eta^*,\eta]_\nu
  \big)
  \in
  \rmS_{2+2\nu}
  =0
\end{gather}
for~$\nu\in\{1,2,3,4,6\}$.
The recursions will follow from calculating the Fourier coefficients of the three summands explicitly.
\section{Computation of Fourier coefficients}
\label{sec:fourier}
\subsection{The spt component}
\label{ssec:fourier:spt}
The first summand of~\eqref{eq:preliminaries:identity} yields the weighted sum involving~$\spt$:
\begin{proposition}
  \label{prop:spt_component}
  We have
  \begin{gather*}
    c([F^+,\eta]_\nu,N)
    =
    \sum_{3m^2+m\le 2N}
    (-1)^m
    \;
    f_\nu(N,m)
    \;
    \spt
    \big(
    N-
    \mfrac{3m^2+m}{2}
    \big)
    \tx{,}
  \end{gather*}
  where
  \begin{gather*}
    f_\nu(N,m)
    =
    \sum_{\mu=0}^\nu
    \;
    (-1)^\mu
    \;
    \mbinom{\nu+\frac{1}{2}}{\nu-\mu}
    \mbinom{\nu-\frac{1}{2}}{\mu}
    \;
    \big(
    N-
    \mfrac{3m^2+m}{2}
    -
    \mfrac{1}{24}
    \big)^\mu
    \;
    \big(
    \mfrac{3m^2+m}{2}
    +
    \mfrac{1}{24}
    \big)^{\nu-\mu}
    \tx{.}
  \end{gather*}
\end{proposition}
\begin{proof}
  Using the Fourier expansion of~$\eta$ given in~\eqref{eq:dedekind_eta} and Definition~\ref{def:rankin-cohen} yields
  \begin{gather*}
    c([F^+,\eta]_\nu,N)
    =
    \sum_{\mu=0}^\nu
    (-1)^\mu
    \mbinom{\nu+\frac{1}{2}}{\nu-\mu}
    \mbinom{\nu-\frac{1}{2}}{\mu}
    \sum_{\substack{n+\frac{3m^2+m}{2}=N \\ n\in\ZZ_{\ge 0}, m \in\ZZ}}
    \hspace{-12pt}
    \big(n-\mfrac{1}{24}\big)^\mu
    \;
    (-1)^m
    \;
    \big(\mfrac{3m^2+m}{2}+\mfrac{1}{24}\big)^{\nu-\mu}
    \spt(n)
    \tx{,}
  \end{gather*}
  and the result follows from switching the sums and setting~$n=N-\frac{3m^2+m}{2}$.
\end{proof}

\subsection{The quasimodular component}
\label{ssec:fourier:quasimodular}
The second summand of~\eqref{eq:preliminaries:identity} is indeed a quasimodular form as a result of the following fact.
\begin{lemma}
  \label{la:quasimodular_quotient}
  $Q_k=\eta^{(k)}/\eta$ is a quasimodular form of weight~$2k$ and depth~$k$.
\end{lemma}
\begin{proof}
  For~$k=1$, we have the well-known identity
  \begin{gather*}
    Q_1=\mfrac{1}{24}E_2
    \tx{.}
  \end{gather*}
  Furthermore, an elementary calculus computation shows that
  \begin{gather*}
    Q_{k}^{(1)}
    =
    {Q_{k+1}-Q_kQ_1}
    \tx{,}
    \quad
    \tx{i.e.}
    \quad
    Q_{k+1}=Q_k^{(1)}+Q_kQ_1
    \tx{,}
  \end{gather*}
  and the claim follows from induction on~$k$ and Propositions~\ref{prop:quasimodular_multiplication} and~\ref{prop:quasimodular_derivative}.
\end{proof}
\begin{proposition}
  \label{prop:quasimodular_component}
  We have
  \begin{align*}
    [E_2/\eta,\eta]_1
    &=
    \mfrac{1}{2}
    E_2^{(1)}
    +
    \mfrac{1}{12}
    E_4
    \tx{,}
    \\
    [E_2/\eta,\eta]_2
    &=
    \mfrac{3}{8}
    E_2^{(2)}
    +
    \mfrac{1}{32}
    E_4^{(1)}
    +
    \mfrac{1}{96}
    E_6
    \tx{,}
    \\
    [E_2/\eta,\eta]_3
    &=
    \mfrac{5}{16}
    E_2^{(3)}
    +
    \mfrac{5}{192}
    E_4^{(2)}
    +
    \mfrac{5}{1728}
    E_6^{(1)}
    +
    \mfrac{5}{3456}
    E_8
    \tx{,}
    \\
    [E_2/\eta,\eta]_4
    &=
    \mfrac{35}{128}
    E_2^{(4)}
    +
    \mfrac{35}{1536}
    E_4^{(3)}
    +
    \mfrac{35}{13824}
    E_6^{(2)}
    +
    \mfrac{35}{110592}
    E_8^{(1)}
    +
    \mfrac{35}{165888}
    E_{10}
    \tx{.}
  \end{align*}
\end{proposition}

\begin{proof}
  An explicit calculation shows that
  \begin{align*}
    [E_2/\eta,\eta]_1
    &=
    2E_2Q_1-\mfrac{1}{2}E_2^{(1)}
    \tx{,}
    \\
    [E_2/\eta,\eta]_2
    &=
    \mfrac{3}{2}E_2Q_2
    +
    \mfrac{9}{2}E_2Q_1^2
    -
    \mfrac{9}{2}E_2^{(1)}Q_1
    +
    \mfrac{3}{8}E_2^{(2)}
    \tx{,}
    \\
    [E_2/\eta,\eta]_3
    &=
    \mfrac{5}{2}E_2Q_3
    +
    \mfrac{5}{2}E_2Q_1Q_2
    -
    10E_2^{(1)}Q_2
    +
    15E_2Q_1^3
    -
    15E_2^{(1)}Q_1^2
    +
    \mfrac{15}{2}
    E_2^{(2)}Q_1
    -
    \mfrac{5}{16}E_2^{(3)}
    \tx{,}
    \\
    [E_2/\eta,\eta]_4
    &=
    \mfrac{35}{16}E_2Q_4
    +
    35E_2Q_1Q_3
    -
    \mfrac{385}{16}E_2^{(1)}Q_3
    -
    \mfrac{525}{16}E_2Q_2^2
    -
    \mfrac{525}{16}E_2^{(1)}Q_1Q_2
    +
    \mfrac{525}{16}E_2^{(2)}Q_2
    \\
    &\quad
    +
    \mfrac{525}{8}E_2Q_1^4
    -
    \mfrac{525}{8}E_2^{(1)}Q_1^3
    +
    \mfrac{525}{16}E_2^{(2)}Q_1^2
    -
    \mfrac{175}{16}E_2^{(3)}Q_1
    +
    \mfrac{35}{128}E_2^{(4)}
    \tx{,}
  \end{align*}
  and using Lemma~\ref{la:quasimodular_quotient} and Propositions~\ref{prop:quasimodular_multiplication} and~\ref{prop:quasimodular_derivative}, we see that~$[E_2/\eta,\eta]_\nu\in\td{M}_{2\nu}^{\le\nu}$ for~$\nu=1,2,3,4$.
  The claim follows from computing the first few coefficients of each expression and applying Proposition~\ref{prop:quasimodular_structure} and computer algebra.
\end{proof}

\subsection{The nonholomorphic component}
\label{ssec:fourier:nonhol}
We recall a fact from~\cite{ortiz-raum-richter-2026}.
\begin{lemma}[\cite{ortiz-raum-richter-2026}, Lemma 3.8]
  \label{la:fourier:nonhol}
  For an integer~$\nu \ge 0$ and rational numbers~$n<0$ and~$\td{n}>0$ with~$\td{n}>|n|$,
  \begin{gather*}
    \pi^\hol_{2+2\nu}\big( \big[
      \Ga\big(-\tfrac{1}{2}, 4 \pi |n| y\big)
      e(n \tau),\,
      e(\td{n} \tau)
      \big]_\nu
    \big)
    =
    2\sqrt{\pi}
    \mbinom{\nu-\frac{1}{2}}{\nu}\,
    \frac{1}{\sqrt{|n|}}
    \big(
    \sqrt{\td{n}}
    -
    \sqrt{|n|}
    \big)
    ^{2\nu+1}
    \,
    e\big( (n+\td{n})\tau \big)
    \tx{.}
  \end{gather*}
\end{lemma}
Applying this specifically to the Fourier expansion of~$\eta$ yields the following:
\begin{proposition}
  \label{prop:nonhol_component}
  We have
  \begin{gather*}
    c\big(
    \pi_{2+2\nu}^\hol
    \big(
    [\eta^*,\eta]_\nu
    \big),
    N\big)
    =
    -\mfrac{4\pi}{\sqrt{6}}
    \big(\mfrac{1}{6}\big)^\nu
    \mbinom{\nu-\frac{1}{2}}{\nu}
    \;
    \lambda_{2\nu+1}^*(N)
    \tx{,}
  \end{gather*}
  where~$\lambda^*_{2\nu+1}(N)$ is as in~\eqref{eq:lambda_star}.
\end{proposition}

\begin{proof}
  From~\eqref{eq:def:non-holomorphic-eichler-integral} and Lemma~\ref{la:fourier:nonhol},
  \begin{align*}
    \pi_{2+2\nu}^\hol
    \big(
    [\eta^*,\eta]_\nu
    \big)
    &=
    -\sqrt{4\pi}
    \sum_{m,\td{m}\in \ZZ}
    (-1)^{m+\td{m}}
    \sqrt{M}
    \cdot
    \pi_{2+2\nu}^\hol
    \big(
    \big[
      \Ga\big(
      -\mfrac{1}{2},
      4\pi My\big)
      e(-{M}\tau),
      e(\td{M}\tau)
      \big]_\nu
    \big)
    \\
    &=
    -4\pi
    \mbinom{\nu-\frac{1}{2}}{\nu}
    \sum_{\substack{m,\td{m}\in\ZZ \\ \td{M}\ge M}}
    (-1)^{m+\td{m}}
    \big(
    \sqrt{\td{M}}
    -
    \sqrt{M}
    \big)^{2\nu+1}
    e\big((\td{M}-M)\tau\big)
    \tx{,}
  \end{align*}
  where~$M=\frac{3m^2+m}{2}+\frac{1}{24}=\frac{1}{24}(6m+1)^2$, and likewise for~$\td{M}$.
  Hence
  \begin{align*}
    c\big(
    \pi_{2+2\nu}^\hol
    \big(
    [\eta^*,\eta]_\nu
    \big),
    N\big)
    &=
    -4\pi
    \mbinom{\nu-\frac{1}{2}}{\nu}
    \sum_{\td{M}-M=N}
    (-1)^{m+\td{m}}
    \big(
    \sqrt{\td{M}}
    -
    \sqrt{M}
    \big)^{2\nu+1}
    \\
    &=
    -\mfrac{2\pi}{\sqrt{6}}
    \big(\mfrac{1}{24}\big)^\nu
    \mbinom{\nu-\frac{1}{2}}{\nu}
    \sum_{\substack{ab=24N \\ m,\td{m}\in \ZZ}}
    (-1)^{m+\td{m}}
    a^{2\nu+1}
    \tx{,}
  \end{align*}
  where~$a=|6\td{m}+1|-|6m+1|$ and~$b=|6\td{m}+1|+|6m+1|$.

  We now inspect the congruence modulo 6 of the divisors~$a$ and~$b$.
  First, take~$m,\td{m}\geq 0$.
  Then~$a\equiv 0\bmod{6}$ and~$b\equiv 2\bmod{6}$.
  Furthermore,~$m+\td{m}\equiv \td{m}-m=\big|\td{m}+\frac{1}{6}\big|-\big|m+\frac{1}{6}\big|=\frac{a}{6}\bmod{2}$.
  Knowing that~$a$ and ~$b$ are both even, we take~$a\ra 2a$ and~$b\ra 2b$ to obtain
  \begin{gather}
    \label{eq:nonhol:i}
    \tag{\tx{I}}
    -\mfrac{2\pi}{\sqrt{6}}
    \big(\mfrac{1}{24}\big)^\nu
    \mbinom{\nu-\frac{1}{2}}{\nu}
    \sum_{\substack{(2a)(2b)=24N \\ a\equiv 0\bmod{3}\\ b\equiv 1\bmod{3} \\ a<b}}
    (-1)^{a/3}
    (2a)^{2\nu+1}
    =
    -\mfrac{4\pi}{\sqrt{6}}
    \big(\mfrac{1}{6}\big)^\nu
    \mbinom{\nu-\frac{1}{2}}{\nu}
    \sum_{\substack{ab=6N \\ a\equiv 0\bmod{3}\\ b\equiv 1\bmod{3} \\ a<b}}
    (-1)^{a/3}
    a^{2\nu+1}
    \tx{.}
  \end{gather}
  Next consider~$m\ge 0$ and~$\td{m}<0$, in which case~$a\equiv4\bmod{6}$ and~$b\equiv 0\bmod{6}$, and~$m+\td{m}\equiv m-\td{m}=\big|m+\frac{1}{6}\big|+\big|\td{m}+\frac{1}{6}\big|=\frac{b}{6}\bmod{2}$.
  Repeating the same procedure as for~\eqref{eq:nonhol:i} yields the contribution
  \begin{gather}
    \label{eq:nonhol:ii}
    \tag{\tx{II}}
    -\mfrac{4\pi}{\sqrt{6}}
    \big(\mfrac{1}{6}\big)^\nu
    \mbinom{\nu-\frac{1}{2}}{\nu}
    \sum_{\substack{ab=6N \\ a\equiv 2\bmod{3}\\ b\equiv 0\bmod{3} \\ a<b}}
    (-1)^{b/3}
    a^{2\nu+1}
    \tx{.}
  \end{gather}
  In the case~$m<0$ and~$m\ge 0$, we have~$a\equiv 2\bmod{6}$ and~$b\equiv 0\bmod{6}$.
  Likewise,~$m+\td{m}\equiv \td{m}-m=\big|\td{m}+\frac{1}{6}\big|+\big|m+\frac{1}{6}\big|=\frac{b}{6}\bmod{2}$.
  Hence this case contributes
  \begin{gather}
    \label{eq:nonhol:iii}
    \tag{\tx{III}}
    -\mfrac{4\pi}{\sqrt{6}}
    \big(\mfrac{1}{6}\big)^\nu
    \mbinom{\nu-\frac{1}{2}}{\nu}
    \sum_{\substack{ab=6N \\ a\equiv 1\bmod{3}\\ b\equiv 0\bmod{3} \\ a<b}}
    (-1)^{b/3}
    a^{2\nu+1}
    \tx{.}
  \end{gather}
  Finally, for~$m,\td{m}<0$, we have~$a\equiv 0\bmod{6}$ and~$b\equiv4\bmod{6}$, so~$m+\td{m}\equiv m-\td{m}=\big|\td{m}+\frac{1}{6}\big|-\big|m+\frac{1}{6}\big|=\frac{a}{6}\bmod{2}$.
  Hence this case yields
  \begin{gather}
    \label{eq:nonhol:iv}
    \tag{\tx{IV}}
    -\mfrac{4\pi}{\sqrt{6}}
    \big(\mfrac{1}{6}\big)^\nu
    \mbinom{\nu-\frac{1}{2}}{\nu}
    \sum_{\substack{ab=6N \\ a\equiv 0\bmod{3}\\ b\equiv 2\bmod{3} \\ a<b}}
    (-1)^{a/3}
    a^{2\nu+1}
    \tx{.}
  \end{gather}
  Combining~\eqref{eq:nonhol:i},~\eqref{eq:nonhol:ii},~\eqref{eq:nonhol:iii}, and~\eqref{eq:nonhol:iv} yields the claimed formula.
\end{proof}

\section{Proof of main theorem}
\label{sec:proof_main}
\begin{proof}[Proof of Theorem~\ref{thm:main_theorem}]
  The recursions follow from inspecting the~$N$th coefficient of~\eqref{eq:preliminaries:identity} using Propositions~\ref{prop:spt_component},~\ref{prop:quasimodular_component}, and~\ref{prop:nonhol_component}.
  For~$\nu=1$, for example,
  \begin{gather*}
    f_1(N,m)
    =
    3m^2+m-\mfrac{1}{2}N+\mfrac{1}{12}
    \tx{.}
  \end{gather*}
  In addition,
  \begin{gather*}
    c([E_2/\eta,\eta]_1,N)
    =
    c\big(\mfrac{1}{2}E_2^{(1)}+\mfrac{1}{12}E_4,N\big)
    =
    -12N\sigma_1(N)+20\sigma_3(N)
    \tx{.}
  \end{gather*}
  Finally,
  \begin{gather*}
    c\big(
    \pi_{2+2\nu}^\hol
    \big(
    [\eta^*,\eta]_\nu
    \big),
    N\big)
    =
    -\mfrac{\pi}{3\sqrt{6}}
    \lambda_3^*(N)
    \tx{.}
  \end{gather*}
  Rearranging and clearing the denominator of 12 yields the first recursion.
  The rest follow analogously.
\end{proof}
  
\vspace{1.5\baselineskip}
\phantomsection
\addcontentsline{toc}{section}{References}
\markright{References}
\label{sec:references}
{
  \sloppy
  \linespread{0.8}
  \printbibliography[heading=none]
}

\filbreak
\Needspace*{5\baselineskip}
\noindent%
\rule{\textwidth}{0.15em}
\\\nopagebreak

{\small\noindent
  Matthew Ortiz\\\nopagebreak
  Department of Mathematics\\\nopagebreak
  University of North Texas\\\nopagebreak
  Denton, TX 76203, USA\\\nopagebreak
  E-mail: \url{matthewortiz2@my.unt.edu}
}
\end{document}

%% file: packages.tex
\usepackage{etoolbox}
\usepackage{ifdraft}
\usepackage{iftex}

\ifluatex
\usepackage[utf8]{luainputenc}
\else
\usepackage[utf8]{inputenc}
\usepackage[T1]{fontenc}
\fi

\usepackage[english]{babel}
\usepackage[autostyle=true]{csquotes}

\usepackage{lmodern}
\usepackage{fourier}

\usepackage{amsfonts}
\usepackage{mathrsfs} 
\usepackage{dsfont}

\usepackage{amssymb}
\usepackage{stmaryrd}    

\usepackage{amsmath}
\usepackage{mathtools}
\usepackage[thmmarks, amsmath, amsthm]{ntheorem}
\usepackage{nccmath}    
\usepackage{tensor}

\usepackage[
  hscale=0.8,vscale=0.875,
  includehead,footskip=2\baselineskip,
  ]{geometry}

\usepackage[draft=false]{scrlayer-scrpage}
\usepackage{needspace}

\usepackage[cmyk,dvipsnames]{xcolor}

\ifdraft{%
\usepackage[textsize=scriptsize,colorinlistoftodos]{todonotes}
\usepackage{lineno}
\usepackage{scrtime}
}{}

\usepackage{booktabs}
\usepackage[inline]{enumitem}
\usepackage{lettrine}
\usepackage{url}

\usepackage[backend=biber,style=numeric-comp,giveninits,url=false,sortcites=true,maxbibnames=99]{biblatex}
\usepackage[final,
            pdfborder={0 0 0}, colorlinks=true,
            linkcolor=BrickRed, citecolor=ForestGreen, urlcolor=RoyalBlue]{hyperref}

%% file: layout.tex
\ifdraft{%
\linenumbers
\newcommand*\patchAmsMathEnvironmentForLineno[1]{%
  \expandafter\let\csname old#1\expandafter\endcsname\csname #1\endcsname
  \expandafter\let\csname oldend#1\expandafter\endcsname\csname end#1\endcsname
  \renewenvironment{#1}%
     {\linenomath\csname old#1\endcsname}%
     {\csname oldend#1\endcsname\endlinenomath}}%
\newcommand*\patchBothAmsMathEnvironmentsForLineno[1]{%
  \patchAmsMathEnvironmentForLineno{#1}%
  \patchAmsMathEnvironmentForLineno{#1*}}%
\AtBeginDocument{%
\patchBothAmsMathEnvironmentsForLineno{equation}%
\patchBothAmsMathEnvironmentsForLineno{align}%
\patchBothAmsMathEnvironmentsForLineno{flalign}%
\patchBothAmsMathEnvironmentsForLineno{alignat}%
\patchBothAmsMathEnvironmentsForLineno{gather}%
\patchBothAmsMathEnvironmentsForLineno{multline}%
}}

\clearmainofpairofpagestyles
\automark[section]{subsection}

\renewcommand{\subsectionmark}[1]{}

\lohead{{\small\normalfont \headertitle}}
\rohead{{\small \headerauthors}}

\RedeclareSectionCommand[%
  font=\Large\sffamily\bfseries,%
  beforeskip=1\baselineskip,%
  afterskip=0.5\baselineskip,%
  indent=0em,%
  afterindent=false,%
  tocbeforeskip=0.3\baselineskip plus 1pt minus 1pt%
  ]{section}

\RedeclareSectionCommands[%
  font=\normalfont\bfseries,%
  beforeskip=3pt,%
  afterskip=-1em%
  ]{subsection,subsubsection}

\RedeclareSectionCommands[%
  font=\normalfont\itshape,%
  beforeskip=.5\baselineskip,%
  afterskip=-1em,%
  indent=0pt,%
]{paragraph}

\AtEveryBibitem{\clearfield{doi}}
\AtEveryBibitem{\clearfield{isbn}}
\AtEveryBibitem{\clearfield{issn}}
\AtEveryBibitem{\clearfield{pages}}
\AtEveryBibitem{\clearlist{language}}

\renewbibmacro*{in:}{}
\DeclareFieldFormat
  [article,inbook,incollection,inproceedings,patent,thesis,unpublished,misc]
  {title}{#1}

\setitemize{itemsep=0.02\baselineskip}

\newenvironment{enumerateroman}{
\begin{enumerate}[label=(\roman*),%
  leftmargin=2.5em,itemindent=0pt,%
  labelindent=.5em,labelwidth=1.5em,labelsep=!,%
  nosep]
}{
\end{enumerate}
}

\newenvironment{enumeratearabic*}{
\begin{enumerate*}[label=(\arabic*)] 
}{
\end{enumerate*}
}

\newenvironment{enumerateroman*}{
\begin{enumerate*}[label=(\roman*)] 
}{
\end{enumerate*}
}

\numberwithin{equation}{section}

\theoremnumbering{arabic}
\newtheorem{theoremcounter}{theoremcounter}[section]
\theoremnumbering{Alph}
\newtheorem{maintheoremcounter}{maintheoremcounter}

\theoremstyle{plain}

\newtheorem{lemma}[theoremcounter]{Lemma}

\newtheorem{proposition}[theoremcounter]{Proposition}

\theoremstyle{plain}

\newtheorem{maintheorem}[maintheoremcounter]{Theorem}

\theoremstyle{definition}

\newtheorem{definition}[theoremcounter]{Definition}

\theoremstyle{remark}

\newtheorem{remark}[theoremcounter]{Remark}

\theoremstyle{nonumberremark}



%% file: shortcuts.tex
\newcommand{\tx}{\text}
\newcommand{\thdash}{\nbd th}
\newcommand{\nbd}{\nobreakdash-\hspace{0pt}}

\newcommand{\fref}[2]{\hyperref[#2]{#1~\ref*{#2}}}

\makeatletter
\newcommand{\writelabel}[1]{#1\def\@currentlabel{#1}}
\makeatother

\newcommand{\minwidthmathbox}[2]{%
  \mathmakebox[{\ifdim#1<\width\width\else#1\fi}]{#2}%
}

\newcommand{\tbf}{\bfseries}

\newcommand{\bbM}{\mathbb{M}}

\newcommand{\cO}{\mathcal{O}}

\newcommand{\rmd}{\mathrm{d}}

\newcommand{\rmM}{\mathrm{M}}

\newcommand{\rmS}{\mathrm{S}}

\newcommand{\td}{\tilde}
\newcommand{\wtd}{\widetilde}
\newcommand{\ov}{\overline}
\newcommand{\ul}{\underline}

\newcommand{\defcol}{\mathrel{:}}
\newcommand{\defeq}{\mathrel{:=}}

\newcommand{\mrelspace}[1]{\mathrel{\mspace{#1}}}

\NewCommandCopy{\rightarroworig}{\rightarrow}
\renewcommand{\rightarrow}
  {\protect\relbar\mrelspace{-9.7mu}\rightarroworig}

\NewCommandCopy{\leftarroworig}{\leftarrow}
\renewcommand{\leftarrow}
  {\protect\leftarroworig\mrelspace{-9.7mu}\relbar}

\newcommand{\ra}{\rightarrow}

\newcommand{\ZZ}{\mathbb{Z}}
\newcommand{\QQ}{\mathbb{Q}}

\newcommand{\CC}{\mathbb{C}}

\newcommand{\SL}[1]{\mathrm{SL}_{#1}}

\newcommand{\HS}{\mathbb{H}}

%% file: shortcuts_this.tex
\newcommand{\Ga}{\Gamma}
\newcommand{\ga}{\gamma}

\newcommand{\hol}{\mathrm{hol}}
\DeclareMathOperator{\Ind}{Ind}
\newcommand{\spt}{\mathrm{spt}}

\newcommand{\intrmd}{\rmd\mspace{-2mu}}